\pdfoutput=1

\documentclass[runningheads]{llncs}
\usepackage{graphicx}

\usepackage{enumerate}
\usepackage[fleqn]{amsmath}
\usepackage{amsfonts}
\usepackage{amssymb}
\usepackage{amsbsy}
\usepackage{paralist, tabularx}

\def\ve#1{\mathchoice{\mbox{\boldmath$\displaystyle\bf#1$}}
	{\mbox{\boldmath$\textstyle\bf#1$}}
	{\mbox{\boldmath$\scriptstyle\bf#1$}}
	{\mbox{\boldmath$\scriptscriptstyle\bf#1$}}}
\newcommand{\N}{\ensuremath{\mathbb{N}}}
\newcommand{\Z}{\ensuremath{\mathbb{Z}}}
\newcommand{\R}{\ensuremath{\mathbb{R}}}

\def\Orthant_j{{\mathcal O}_{j}}

\newcommand\veb{{\ve b}}

\newcommand\vel{{\ve l}}

\newcommand\veu{{\ve u}}

\newcommand\vew{{\ve w}}
\newcommand\vex{{\ve x}}
\newcommand\vey{{\ve y}}

\begin{document}
\title{Block-structured Integer Programming: Can we Parameterize without the Largest Coefficient?\thanks{Research was supported in part by NSF 1756014 and NSFC 11531014.}}
\titlerunning{Block-structured IP: Parameterizing without the Largest Coefficient}
%
\author{Lin Chen\inst{1} \and
Hua Chen\inst{2} \and
Guochuan Zhang\inst{2}}
\authorrunning{L. Chen et al.}
%
\institute{Texas Tech University, Lubbock, TX, US\\
	\email{chenlin198662@gmail.com}
	\and
	Zhejiang University, Hangzhou, China\\
	\email{chenhua\_by@zju.edu.cn; zgc@zju.edu.cn}
	}
\maketitle              
\begin{abstract}
We consider 4-block $n$-fold integer programming, which can be written as $\max\{\vew\cdot \vex: H \vex=\veb, \vel\le \vex\le \veu, \vex\in \Z^{N} \}$ where the constraint matrix $H$ is composed of small submatrices $A,B,C,D$ such that the first row of $H$ is $(C,D,D,\cdots,D)$, the first column of $H$ is $(C,B,B,\cdots,B)$, the main diagonal of $H$ is $(C,A,A,\cdots,A)$, and all the other entries are $0$. The special case where $B=C=0$ is known as $n$-fold integer programming.

Prior algorithmic results for 4-block $n$-fold integer programming and its special cases usually take $\Delta$, the largest absolute value among entries of $H$ as part of the parameters. In this paper, we explore the possibility of getting rid of $\Delta$ from parameters, i.e., we are looking for algorithms that runs polynomially in $\log\Delta$. We show that, assuming $\text{P}\neq \text{NP}$, this is not possible even if $A=(1,1,\Delta)$ and $B=C=0$. However, this becomes possible if $A=(1,1,\cdots,1)$ or $A\in \Z^{1\times 2}$, or more generally if $A\in\Z^{s_A\times t_A} $ where $t_A=s_A+1$ and the rank of matrix $A$ satisfies that $\text{rank}(A)=s_A$. More precisely,   

\begin{itemize}
	\item	If $A=(1,\ldots,1)\in \Z^{1\times t_A} $, then 4-block $n$-fold IP can be solved in $(t_A+t_B)^{O(t_A+t_B)}\cdot poly(n,\log\Delta)$ time;
	\item If $A\in\Z^{s_A\times t_A} $, $t_A=s_A+1$ and $\text{rank}(A)=s_A$, then 4-block $n$-fold IP can be solved in $(t_A+t_B)^{O(t_A+t_B)}\cdot n^{O(t_A)}\cdot poly(\log\Delta)$ time; Specifically, if in addition we have $B=C=0$ (i.e., $n$-fold integer programming), then it can be solved in linear time $n\cdot poly(t_A,\log \Delta)$.
\end{itemize}

\end{abstract}
\clearpage

\section{Introduction}
\label{sec:typesetting-summary}

\emph{Integer Programming} is widely used as a modelling tool for a variety of combinatorial optimization problems.  A standard form of an integer program (IP) is defined as follows:
\begin{eqnarray}\label{ILP}
\max\{\vew\cdot \vex: H \vex=\veb, \vel\le \vex\le \veu, \vex\in \Z^{N} \}
\end{eqnarray}
where the coordinates of $H,\vew,\veb,\vel,\veu$ are integers. Here $H$ is the \emph{constraint matrix} with dimension $M\times N$. We let $\Delta$ be the largest absolute value among all the entries of $H$. 

In general, IP is NP-hard, which was shown by Karp~\cite{karp1972reducibility}, thus motivating the search for tractable special cases. There are two important lines of research in the literature which target at different parameters and motivate our research in this paper. 
The first line of research dated back to the work of Papadimitriou in 1981~\cite{papadimitriou1981complexity}, where he considered IPs with few constraints, and   
provided an algorithm whose time is $(M\cdot \Delta)^{O(M^2)}$. 
This result was later improved by Eisenbrand and Weismantel~\cite{eisenbrand2019proximity}, and then by Jansen et al.~\cite{jansen2018integer}. So far the best known result is $(\sqrt{M}\Delta)^{O(M)}\cdot\log(\|\veb\|_{\infty})$, where $\|\veb\|_{\infty}$ represents the maximal absolute value of coordinates in vector $\veb$. 
The second line of research dated back to the work of Lenstra~\cite{lenstra1983integer} in 1983, where he considered IPs with few variables. This result was later on improved by Kannan~\cite{kannan1987minkowski} who presented an algorithm of running time $N^{O(N)}\cdot poly(M,\log\Delta)$. In recent years, there is further improvement on the coefficient of the exponent in the term $N^{O(N)}$ (see, e.g.~\cite{dadush2011enumerative}).


The above algorithms require $H$ to have either few rows or few columns, but in many applications it may be inevitable to have a constraint matrix with a huge number of rows and columns. In recent years, there is an increasing interest in the study of IP where the constraint matrix $H$ may have many rows and columns, but has a more restricted block structure. Such block-structured IP finds application in a variety of optimization problems including string matching, computational social choice, resource allocation, etc (see ,e.g.~\cite{knop2019combinatorial,faliszewski2018opinion,knop2020voting,chen2018covering,jansen2018empowering,knop2018scheduling}). We give a brief introduction below.
\smallskip
\noindent\textbf{Block-structured IP.}
We consider IP~\eqref{ILP} where $H$ is built from small submatrices $A$, $B$, $C$ and $D$ in the following form:
\begin{eqnarray}\label{eq:4block}
H=
\begin{pmatrix}
C & D & D & \cdots & D \\
B & A & 0  &   & 0  \\
B & 0  & A &   & 0  \\
\vdots &   &   & \ddots &   \\
B & 0  & 0  &   & A
\end{pmatrix}. 
\end{eqnarray}
Here, $A,B,C,D$ are $s_i\times t_i$ matrices, where $i=A,B,C,D$, respectively. $H$ consists of $n$ copies of $A,B,D$ and one copy of $C$. Consequently, $N=t_B+nt_A$ and $M=s_C+ns_B$. Notice that by plugging $A,B,C,D$ into the above block structure we require that $s_C=s_D$, $s_A=s_B$, $t_B=t_C$ and $t_A=t_D$. 

The above IP is called 4-block $n$-fold IP. As a special case, when $C=B=0$, it is called $n$-fold IP; when  $C=D=0$, it is called two stage-stochastic IP. It is worth mentioning that recently researchers have also considered more generalized IPs where the submatrices $A,B,D$ are not necessarily identical (i.e., the $n$ identical $A$'s, $B$'s, $D$'s are replaced with $A_i,B_i,D_i$, respectively). We call it generalized 4-block $n$-fold IP, and its two special cases generalized $n$-fold IP and generalized two stage-stochastic IP.  

\smallskip
\noindent\textbf{Related work on Block-structured IP.}
Let $\varphi$ be the encoding length of a block-structured IP. For $n$-fold IP, Hemmecke et al.~\cite{hemmecke2013n} showed an algorithm of running time $n^3t_A^3\varphi\cdot (s_Ds_A\Delta)^{\mathcal{O}(t_A^2s_D)}$. Later on, improved algorithms were developed by a series of researchers including Eisenbrand et al.~\cite{eisenbrand2018faster,eisenbrand2019algorithmic}, Altmanov{\'a} et al.~\cite{altmanova2019evaluating}, Jansen et al.~\cite{jansen2019near}, Cslovjecsek et al.~\cite{cslovjecsek2020n}. So far, generalized $n$-fold IP can be solved in $(s_Ds_A\Delta)^{O(s_A^2+s_As_D^2)} nt_A$. Specifically, if $A=(1,\ldots,1)$ in an $n$-fold IP, then this is called combinatorial $n$-fold IP. Even such a restricted class of IP finds applications in a variety of problems including computational social choice, stringology, etc.~\cite{knop2019combinatorial}.

For two-stage stochastic IP, Hemmecke and Schultz~\cite{hemmecke2003decomposition} were the first to present an algorithm of running time $poly(n)\cdot f(s_A,s_B,t_A,t_B,\Delta)$ for some computable function
$f$, despite that the function $f$ is unknown. Very recently, Klein~\cite{klein2020complexity} developed an algorithm of such a running for generalized two-stage stochastic IP where $f$ is a doubly exponential function.

For 4-block $n$-fold IP, Hemmecke et al.~\cite{hemmecke2010polynomial} gave an algorithm which runs in time $n^{g(s_D+s_A,t_B+t_A,\Delta)}\varphi$ for some computable function $g$ which is doubly exponential. Very recently, Chen et al.~\cite{chen2020new} presented an improved algorithm whose running time is singly exponential.

It is noticeable that early algorithms for $n$-fold IP has a running time exponential in both the number of rows and columns of the small submatrices~\cite{hemmecke2013n}, and recent progress is able to reduce the running time such that it is only exponential in the number of rows of submatrices, coinciding the running time of ``Papadimitrious's line" of algorithm for general IP. It is thus natural to ask, can we hope for a ``Lenstra's line" of algorithm for block-structured IP that is polynomial in $\log\Delta$? More precisely, can we expect an algorithm for block-structured IP of running time $f(s_A,s_B,s_C,s_D,t_A,t_B,t_C,t_D)poly(n,\log\Delta)$, or $(n\log\Delta)^{f(s_A,s_B,s_C,s_D,t_A,t_B,t_C,t_D)}$ if the former is not possible? This paper aims at a systematic study in this direction. 

\smallskip
\noindent\textbf{Our contributions.}
The major contribution of this paper is to give a full characterization on when FPT or XP algorithm exists for block-structured IP without $\Delta$, the largest coefficient, being part of the parameters.

We show that, in general, $n$-fold IP is NP-hard if $\Delta$ does not belong to the parameters. In particular, NP-hardness follows even if the submatrix $A=[1,1,\Delta]$.



On the positive side, we achieve the following algorithmic results:
\begin{itemize}
	\item	If $A=(1,\ldots,1)\in \Z^{1\times t_A} $, then 4-block $n$-fold IP can be solved in $(t_A+t_B)^{O(t_A+t_B)}\cdot poly(n,\log\Delta)$ time;
	\item If $A\in\Z^{s_A\times t_A} $, $t_A=s_A+1$ and $\text{rank}(A)=s_A$, then 4-block $n$-fold IP can be solved in $(t_A+t_B)^{O(t_A+t_B)}\cdot n^{O(t_A)}\cdot poly(\log\Delta)$ time; Specifically, $n$-fold IP can be solved in linear time $n\cdot poly(t_A,\log \Delta)$.
\end{itemize}
It is remarkable that our NP-hardness results already rule out an algorithm of running time $n^{f(t_A)}poly(\log\Delta)$ even for $n$-fold IP when $t_A\ge s_A+2$, hence an algorithm for $t_A=s_A+1$ is the best we can hope for.

One implication of our results is on the impact of the box constraint $\vel\le\vex\le\veu$ to the complexity of block-structured IP. Our NP-hardness result can be translated to the NP-hardness of the following scheduling problem: given $m$ identical machines and three types of jobs, each type of a job has the same processing time on every machine. Each machine $i$ has cardinality constraints such that it can accept at most $c_i^j$ jobs of type $j$ where $j=1,2,3$. The goal is to find an assignment of jobs to machines such that makespan (largest job completion time) is minimized. Note that, however, this scheduling problem is polynomial time solvable if there is no cardinality constraints~\cite{goemans2014polynomiality}. When formulating the scheduling problem using $n$-fold IP, the cardinality constraints hide in the box constraints $\vel\le\vex\le \veu$. Therefore, if we look at the $n$-fold IP formulation of the scheduling problem, a simpler box constraint $\vex\ge 0$ allows a polynomial time algorithm for three or even a constant number of different types of jobs, while a general box constraint $\vel\le\vex\le \veu$ only leads to polynomiality of two types of jobs. The reader will also see that the most technical part of our algorithm lies on the dealing of the box constraints. In contrast, essentially all existing algorithms for block-structured IP rely on an iterative augmentation framework which does not really distinguish between different kinds of box constraints. From that perspective, our algorithmic results can be viewed as a complement to existing algorithms. It remains as an important problem what kind of box constraints can lead to polynomial time algorithms when $t_A\ge s_A+2$.

\section{Preliminaries}
\noindent\textbf{Notation.}
We write vectors in boldface, e.g. $\vex, \vey$, and their entries in normal font, e.g. $x_i, y_i$.
Recall that a solution $\vex$ for $4$-block $n$-fold IP is a $(t_B+nt_A)$-dimensional vector, we write it into $n+1$ \emph{bricks}, such that $\vex=(\vex^0,\vex^1,\cdots,\vex^n)$ where $\vex^0 \in \Z^{t_B}$ and each $\vex^i \in \Z^{t_A}$, $1\le i\le n$. We call $\vex^i$ the \emph{$i$-th brick} for $0\le i\le n$. For a vector or a matrix, we write $\|\cdot\|_{\infty}$ to denote the maximal absolute value of its elements. For two vectors $\vex,\vey$ of the same dimension, $\vex\cdot\vey$ denotes their inner product. 
We use $\text{gcd}(\cdot,\cdot)$ to represent the greatest common divisor of two integers. For example, $\text{gcd}(\lambda,\mu)$ represents the greatest common divisor of integers $\lambda$ and $\mu$. 
We usually use lowercase letters for variables and uppercase letters for matrices. For an arbitrary matrix $H$, we use $\text{rank}(H)$ to denote its rank. We use $poly(x)$ to denote a polynomial in $x$. 

\smallskip
\noindent\textbf{Input size.} In an IP~\eqref{ILP}, it is allowed that the entries of $\veb, \vel,\veu$ are $\infty$. However, utilizing the techniques of Tardos~\cite{tardos1986strongly}, Koutecký et al.~\cite{koutecky2018parameterized} showed that without loss of generality we can restrict that $\|\veb\|_\infty, \|\vel\|_\infty, \|\veu\|_\infty\le 2^{O(n\log n)}\Delta^{O(n)}$. We assume this bound throughout this paper.


\smallskip
\noindent\textbf{B\'{e}zout's identity.}
Let $\lambda$ and $\mu$ be integers with greatest common divisor $\text{gcd}(\lambda,\mu)$. Then, there exist integers $x$ and $y$ such that $\lambda x + \mu y = \text{gcd}(\lambda,\mu)$. 

\smallskip
\noindent\emph{Structure of solutions.}
When an arbitrary solution $(\hat{x}, \hat{y})$ has been computed (e.g., using extended Euclidean algorithm), all pairs of solutions can be represented in the form
$\Big(\hat{x}+\ell{\frac {\mu}{\text{gcd}(\lambda,\mu)}}, \hat{y}-\ell{\frac {\lambda}{\text{gcd}(\lambda,\mu)}}\Big),$
where $\ell$ is an arbitrary integer.

\noindent\textbf{Smith normal form.}
Let $A$ be a nonzero $s\times t$ matrix over a principal ideal domain. $\bar{A}$ is called the Smith normal form of $A$: there exist invertible $s\times s$ and $(t\times t)$-matrices $U$, $V$ such that the product $UAV$ is $\bar{A}$, and its diagonal elements $\alpha_{i}$ satisfy $\alpha_i|\alpha_{i+1}$ for all $1\le i\le h-1$, where $h=\text{rank}(A)$. The rest elements in $\bar{A}$ are zero.

\noindent\emph{Remark.} The process of transforming an integer matrix into its Smith normal form is in polynomial time, i.e., $poly(s,t,\log \Delta)$~\cite{kannan1979polynomial}.

\section{Hardness results}

Recall $n$-fold IP is a special case of $4$-block $n$-fold IP when $B=C=0$ in Eq~\eqref{eq:4block}. The goal of this section is to prove the following theorem.
\begin{theorem}\label{thm:np-nfold}
	It is NP-hard to determine whether an $n$-fold IP admits a feasible solution even if $A=(1,1,\Delta)$ and $D=(1,0,0)$, where $\Delta\in \Z$ is part of the input. 
\end{theorem}
\begin{proof}
	We reduce from subset-sum. In a subset-sum problem, given are $n$ positive integers $\beta_1,\beta_2,\cdots,\beta_n$, and the goal is to find a subset of these integers which add up to exactly $\Delta\in\N$. 
	
	Given a subset-sum instance, we construct an $n$-fold integer program instance such that $A=(1,1,\Delta)$ and $D=(1,0,0)$. Note that each brick $\vex^i=(x^i_1,x^i_2,x^i_3)$. Let the interval constraints for variables be $0\le x^i_1\le \beta_i$, $0\le x^i_2\le \Delta-\beta_i$ and $0\le x^i_3\le 1$. Let $\veb^0=\veb^i=\Delta$. This finishes the construction.
	
	Now we write down explicitly the $n$-fold integer program as follows:
	\begin{subequations}
		\begin{eqnarray}
		&& \sum_{i=1}^n x^i_1=\Delta \label{np-1}\\
		&& x^i_1+x^i_2+\Delta x^i_3=\Delta, \hspace{37mm} \forall 1\le i\le n \label{np-2}\\
		&&  0\le x^i_1\le \beta_i, 0\le x^i_2\le \Delta-\beta_i, 0\le x^i_3\le 1,  \hspace{5mm}\forall 1\le i\le n \nonumber\\
		&& x^i_1, x^i_2, x^i_3\in \Z, \hspace{47mm}  \forall 1\le i\le n \nonumber
		\end{eqnarray}
	\end{subequations}
	
	Since $x^i_3\in\{0,1\}$, there are two possibilities. If $x^i_3=1$, then $x^i_1=x^i_2=0$; otherwise, $x^i_1+x^i_2=\Delta$. As $x^i_1\le \beta_i$ and $x^i_2\le \Delta-\beta_i$, we have $x^i_1=\beta_i$ and $x^i_2=\Delta-\beta_i$ if $x^i_3=0$. Hence, $x^i_1$ is either $0$ or $\beta_i$. By Constraint~\eqref{np-1}, the constructed $n$-fold integer program instance admits a feasible solution if and only if there exists a subset of $\{\beta_1,\beta_2,\cdots,\beta_n\}$ whose sum is $\Delta$. Hence, $n$-fold IP is NP-hard even if $s_D=s_A=1$, and $t_A= 3$.  \qed
\end{proof}
\noindent\textbf{Remark.} Theorem~\ref{thm:np-nfold} also implies the NP-hardness of the following scheduling problem. There are $n$ machines and three types of jobs. The 1st and 2nd type of jobs have a processing time of 1, and the 3rd type of jobs have a processing time of $\Delta$. Each machine $i$ can accept at most $\beta_i$ jobs of type $1$, $\Delta-\beta_i$  jobs of type $2$, and $1$ job of type 3. Given $\Delta$ jobs of type 1, $(n-k-1)\Delta$ jobs of type $2$ and $k$ jobs of type 3, is it possible to schedule all the jobs within makespan $\Delta$? Let $x_j^i$ be the number of jobs of type $j\in\{1,2,3\}$ on machine $i$, we can establish a similar IP as that in the proof of Theorem~\ref{thm:np-nfold} and the NP-hardness follows directly. 

Enforcing dummy constraints, we have the following corollary.
\begin{corollary}
It is NP-hard to determine whether an $n$-fold IP admits a feasible solution if $A\in\Z^{s_A\times t_A} $ and $t_A\ge s_A+2$.\label{n-15}
\end{corollary}

We remark that if we further consider generalized $n$-fold IP where the first row is $(D_1,D_2,\cdots,D_n)$ and the lower diagonal is $(A_1,A_2,\cdots,A_n)$, then essentially all non-trivial cases become NP-hard as is implied by the following theorem. Therefore, we restrict our attention to the standard 4-block $n$-fold IP in this paper.

\begin{theorem}\label{th2}
	It is NP-hard to determine whether a generalized $n$-fold IP admits a feasible solution even if one of the following holds:
	\begin{compactitem}
	    \item $A_i=A=(\Delta,1)$, $D_i=(\beta_i,0)$; or
	    \item $A_i=(1,\beta_i)$, $D_i=D=(1,0)$.
	\end{compactitem}
\end{theorem}



Using a slight variation of the reduction we used in Theorem~\ref{thm:np-nfold}, we can show Theorem~\ref{th2}.

\paragraph{Proof of Theorem~\ref{th2}.}
\begin{compactitem}
	    \item 

We reduce from subset-sum. In a subset-sum problem, given are $n$ positive integers $\beta_1,\beta_2,\cdots,\beta_n$, and the goal is to find a subset of these integers which add up to exactly $\Delta\in\N$. 
	
	Given a subset-sum instance, we construct an $n$-fold integer program instance such that $A=(\Delta,1)$ and $D_i=(\beta_i,0)$. Note that each brick $\vex^i=(x^i_1,x^i_2)$. Let the interval constraints for variables be $0\le x^i_1\le 1$, and $0\le x^i_2\le \Delta$. Let $\veb^0=\veb^i=\Delta$. This finishes the construction.
	
	Now we write down explicitly the generalized $n$-fold integer program as follows:
	\begin{subequations}
	\begin{eqnarray}
	&& \sum_{i=1}^n \beta_ix^i_1=\Delta \label{npp-1}\\
	&& \Delta x^i_1+x^i_2=\Delta, \hspace{38mm} \forall 1\le i\le n \\
	&&  0\le x^i_1\le 1, 0\le x^i_2\le \Delta,   \hspace{23mm}\forall 1\le i\le n \nonumber\\
	&& x^i_1, x^i_2\in \Z, \hspace{43mm}  \forall 1\le i\le n \nonumber
	\end{eqnarray}
	\end{subequations}
	
	Since $x^i_1\in\{0,1\}$, by Constraint~\eqref{npp-1}, we know that  the constructed $n$-fold integer program instance admits a feasible solution if and only if there exists a subset of $\{\beta_1,\beta_2,\cdots,\beta_n\}$ whose sum is $\Delta$. Hence, the generalized $n$-fold IP is NP-hard even if  $A\in \Z^{1\times 2}$.\qed
	\item 
	We still reduce from subset-sum. In a subset-sum problem, given are $n$ positive integers $\beta_1,\beta_2,\cdots,\beta_n$, and the goal is to find a subset of these integers which add up to exactly $\Delta\in\N$. 
	
	Given a subset-sum instance, we construct an $n$-fold integer program instance such that $A_i=(1,\beta_i)$ and $D=(1,0)$. Each brick $\vex^i=(x^i_1,x^i_2)$. Let the interval constraints for variables be $0\le x^i_1\le \beta_i$, and $0\le x^i_2\le 1$. Let $\veb^0=\Delta$ and $\veb^i=\beta_i$. This finishes the construction.
	
	Now we write down the generalized $n$-fold integer program as follows:
	\begin{subequations}
	\begin{eqnarray}
	&& \sum_{i=1}^n x^i_1=\Delta \label{npp-2}\\
	&& x^i_1+\beta_ix^i_2=\beta_i, \hspace{36mm} \forall 1\le i\le n \\
	&&  0\le x^i_1\le \beta_i, 0\le x^i_2\le 1,   \hspace{23mm}\forall 1\le i\le n \nonumber\\
	&& x^i_1, x^i_2\in \Z, \hspace{43mm}  \forall 1\le i\le n \nonumber
	\end{eqnarray}
	\end{subequations}	
	
	We know $x^i_2\in\{0,1\}$, when $x^i_2=0$, $x^i_1=\beta_i$; when $x^i_2=1$, $x^i_1=0$. Combining with Constraint~\eqref{npp-2}, we know that the constructed $n$-fold integer program instance admits a feasible solution if and only if there exists a subset of $\{\beta_1,\beta_2,\cdots,\beta_n\}$ whose sum is $\Delta$. Hence, the generalized $n$-fold IP is NP-hard even if  $A\in \Z^{1\times 2}$.\qed
\end{compactitem}

\section{Algorithms for $4$-block $n$-fold IP}

We complement our hardness results in Theorem~\ref{thm:np-nfold} by establishing algorithms for the following two cases: i). $A=(1,1,\cdots,1)\in \Z^{1\times t_A}$, i.e., $A$ is a $t_A$-dimensional vector that only consists of $1$; ii). $A\in\Z^{1\times 2} $, i.e., $A$ is a vector of dimension 2. We will further generalize the second case to $A\in\Z^{s_A\times t_A} $ where $t_A=s_A+1$ and $\text{rank}(A)=s_A$.

\subsection{The case of $A=(1,1,\cdots,1)$}
The goal of this subsection is to prove the following theorem.
\begin{theorem}\label{11-n}
	If $A=(1,\ldots,1)\in \Z^{1\times t_A} $, then $4$-block $n$-fold IP can be solved in time $(t_A+t_B)^{O(t_A+t_B)}\cdot poly(n,\log\Delta)$.
\end{theorem}

\begin{proof}
	We write the $4$-block $n$-fold IP explicitly as follows: 
	\begin{eqnarray}
	(\text{IP}_1): &\max&	 \vew\vex\nonumber\\
	&& C\vex^0+D\sum_{i=1}^{n}\vex^i=\veb^0 \nonumber\\
	&&B\vex^0+(1,\ldots,1)\vex^i=\veb^i, \hspace{24mm}\forall 1\le i \le n\nonumber \\
	&&\vel^i \le	\vex^i\le \veu^i, \hspace{42mm}\forall 0\le i \le n\nonumber\\
	&&	\vex^0\in \Z^{t_B},	\vex^i\in \Z^{t_A}\hspace{33mm}\ \forall 1\le i \le n\nonumber
	\end{eqnarray}
	In what follows, we show that the above $(\text{IP}_1)$ is equivalent to the following mixed integer linear programming (MIP$_2$) which can be solved in FPT time.
	\begin{subequations}
		\begin{eqnarray*}
			(\text{MIP}_2):	&\max&	 \vew\vex\nonumber\\
			&&	\sum_{i=1}^{n}\vex^i=\vey \nonumber\\
			&&	C\vex^0+D\vey=\veb^0 \nonumber \\
			&&B\vex^0+	(1,\ldots,1)\vex^i=\veb^i, \hspace{22mm}\ \forall 1\le i \le n\nonumber\\
			&&	\vel^i \le	\vex^i\le \veu^i, \hspace{40mm}\ \forall 0\le i \le n\nonumber\\
			&&	\vey\in \Z^{t_A}, \vex^0\in \Z^{t_B}\nonumber\\
			&&	\vex^i\in \R^{t_A}\hspace{46.5mm}\ \forall 1\le i \le n
		\end{eqnarray*}
	\end{subequations}
	Notice that in $(\text{MIP}_2)$ we have $\vex^i\in \R^{t_A}$, whereas there are only $t_A+t_B$ integral variables in total. Applying Kannan's algorithm~\cite{kannan1987minkowski}, the optimal solution $(\vex_*,\vey_*)$ to $(\text{MIP}_2)$ can be computed in $(t_A+t_B)^{O(t_A+t_B)}\cdot poly(n,\log\Delta)$ time. 
	
	Next we show that the optimal solution to $(\text{IP}_1)$ can be derived in polynomial time based on $(\vex_*,\vey_*)$. Notice that in $(\vex_*,\vey_*)$, each brick $\vex_*^i$ may take fractional values, however, we can round them to integral values through the following LP:
	\begin{subequations}
		\begin{eqnarray}
		(\text{LP}_3):	&\max& \vew^0\vex_*^0+\sum_{i=1}^{n} \vew^i\vex^i\nonumber\\
		&&	\sum_{i=1}^{n}\vex^i=\vey_* \label{n-7}\\
		&&B\vex_*^0+	(1,\ldots,1)\vex^i=\veb^i, \hspace{23mm}\ \forall 1\le i \le n\label{n-8}\\
		&&	\vel^i \le	\vex^i\le \veu^i, \hspace{41mm}\ \forall 1\le i \le n\label{n-9}\\
		&&	\vex^i\in \R^{t_A}\hspace{48mm}\ \forall 1\le i \le n\nonumber
		\end{eqnarray}
	\end{subequations}
	Note that $(\text{LP}_3)$ is the linear program by plugging $\vex^0=\vex_*^0$ and $\vey=\vey_*$ into $(\text{MIP}_2)$, hence $\vex^i=\vex^i_*$ is an optimal solution to $(\text{LP}_3)$. Meanwhile, it is not difficult to see that $(\text{LP}_3)$ is essentially an LP for assignment problem, which is totally unimodular \cite{hoffman2010integral}. Hence an integral optimal solution $\vex^i=\bar{\vex}^i$ to $(\text{LP}_3)$ can be computed in $O(n^2t_A+nt_A^2)$ time (see, e.g., Theorem 11.2 in~\cite{korte2018combinatorial}) and it achieves the same objective value as the fractional optimal solution $\vex^i=\vex^i_*$. Therefore, $(\vex_*^0,\bar{\vex}^i,\vey_*)$ is also an optimal solution to $(\text{MIP}_2)$. Overall, we solve $(\text{MIP}_2)$, and hence $(\text{IP}_1)$ in $(t_A+t_B)^{O(t_A+t_B)}\cdot poly(n,\log\Delta)$ time, and Theorem~\ref{11-n} is proved.\qed
\end{proof}

As a corollary, we obtain similar result for $n$-fold IP:
\begin{corollary}
	For $n$-fold IP with $A=(1,\ldots,1)\in \Z^{1\times t_A} $, there exists an FPT algorithm of running time $t_A^{O(t_A)}\cdot poly(n,\log\Delta)$.\label{12-n}
\end{corollary}
\subsection{The case of $A\in\Z^{s_A\times t_A}$, $t_A=s_A+1$ and $\text{rank}(A)=s_A$}
The goal of this subsection is to prove the following theorem.
\begin{theorem}
	If $A\in\Z^{s_A\times t_A} $ and $t_A=s_A+1$ and $\text{rank}(A)=s_A$, then $4$-block $n$-fold IP can be solved in time of $(t_A+t_B)^{O(t_A+t_B)}\cdot n^{O(t_A)}\cdot poly(\log\Delta)$. \label{thmm}
\end{theorem}

Towards this, we start with the simpler case $A\in\Z^{1\times 2}$ to illustrate the main techniques.
\begin{theorem} If $A\in\Z^{1\times 2}$, then $4$-block $n$-fold IP can be solved in time of $t_B^{O(t_B)}\cdot poly(n,\log\Delta)$. \label{n-10}
\end{theorem}

\begin{proof}
	Let $A=(\lambda,\mu)$, we write the constraints of 4-block $n$-fold IP explicitly as follows: 
	\begin{subequations}
		\begin{eqnarray}
		&& C\vex^0+D\sum_{i=1}^{n}\vex^i=\veb^0\label{n-16} \\
		&&B\vex^0+\lambda x_1^i+\mu x_2^i=\veb^i,\hspace{26mm}\forall 1\le i \le n\label{n-17} \\
		&&\vel^i \le	\vex^i\le \veu^i, \hspace{42.5mm}\forall 0\le i \le n\nonumber
		\end{eqnarray}
	\end{subequations}
	
	\noindent\textbf{Step 1. Use the B\'{e}zout's identity to simplify~\eqref{n-16} and~\eqref{n-17}.}
	
	We subtract $B\vex^0+\lambda x^1_1+\mu x^1_2=\veb^1$ from both sides of Eq~\eqref{n-17}, and get the following: $\lambda(x^i_1-x^1_1)+\mu(x^i_2-x^1_2)=\veb^i-\veb^1.$ Then we let $\theta_1=\frac{\mu}{\text{gcd}(\lambda,\mu)}, \theta_2=-\frac{\lambda}{\text{gcd}(\lambda,\mu)},$ where recall $\text{gcd}(\lambda,\mu)$ represents the greatest common divisor of $\lambda$ and $\mu$. 
	According to the B\'{e}zout's identity, we can get the following general solution: 
	\begin{eqnarray}
	&& x_h^i=\hat{x}^i_{h}+\theta_h y_i +x_h^1,\quad h=1,2, i=2,3,\cdots, n\label{nn3}
	\end{eqnarray}
	where $(\hat{x}^i_{1},\hat{x}^i_{2})$ is an arbitrary solution to  $\lambda\hat{x}^i_{1}+\mu\hat{x}^i_{2}=\veb^i-\veb^1$.  To be consistent, we introduce dummy variables $\hat{x}_h^1=0$ for $h=1,2$ and $y_1=0$, whereas Eq~\eqref{nn3} also holds for $i=1$.
	
	Notice that from now on $\theta_h$, $\hat{x}_h^i$ are all fixed values.
	
	By Eq~\eqref{nn3}, we have 
	\begin{eqnarray}
	&& \sum_{i=1}^{n}x_h^i=\sum_{i=1}^{n}\hat{x}^i_{h}+\theta_h\sum_{i=1}^{n} y_i+nx_h^1,\quad h=1,2 \nonumber
	\end{eqnarray}
	Plug the above into Eq~\eqref{n-16}, we have
	\begin{eqnarray}
	&&	C\vex^0+ D \left(
	\begin{array}{c}
	\sum_{i=1}^{n}\hat{x}^i_{1}+\theta_1\sum_{i=1}^{n}y_i+nx_1^1\\
	\sum_{i=1}^{n}\hat{x}^i_{2}+\theta_2\sum_{i=1}^{n}y_i +nx_2^1\\
	\end{array} \right) =\veb^0.\label{nn1}
	\end{eqnarray}
	Till now, we have transformed 4-block $n$-fold IP into an equivalent IP with variables $y_i$ and $x^1_h$ for	$1\le i\le n$ and $h=1,2$. 
	
	Next, we divide $x^1_h$ by $\theta_h$ and denote by $\xi_h$ and $z_h$ its remainder and quotient, respectively, that is,
	\begin{eqnarray}
	&& x_h^1=\xi_h+\theta_h z_h, \quad h=1,2,\label{nn2}
	\end{eqnarray}
	where $\xi_h\in [0,|\theta_h|-1]$. 
	
	Now we can rewrite the 4-block $n$-fold IP using new variables $\xi_h,z_h$ (where $h=1,2$) and $y_i$ (where $1\le i\le n$). 
	\begin{subequations}
		\begin{eqnarray}
		(\text{IP}_4):	&\max& \vew\vex=\vew^0\vex^0+c_0+\sum_{h=1}^2\sum_{i=1}^{n}w^i_h\xi_h+\sum_{h=1}^2\sum_{i=1}^{n}[w^i_h\theta_h(y_i+z_h)]\nonumber\\
		&&
		C\vex^0+ D \left(
		\begin{array}{c}
		\sum_{i=1}^{n}\hat{x}^i_{1}+\theta_1\sum_{i=1}^{n}y_i+n(\xi_1+\theta_1z_1)\\
		\sum_{i=1}^{n}\hat{x}^i_{2}+\theta_2\sum_{i=1}^{n}y_i +n(\xi_2+\theta_2z_2)\\
		\end{array} \right) =\veb^0  \\
		&&	B\vex^0+\lambda\xi_1+\mu\xi_2+\lambda z_1\theta_1+\mu z_2 \theta_2=\veb^1  \\
		&&	y_1=0  \\
		&&\vel^i \le	\vex^i\le \veu^i, \hspace{46mm}\forall 0\le i \le n \label{IP4:box}
		\end{eqnarray}
	\end{subequations}
	where $c_0:=\sum_{i=1}^{n}(w^i_1\hat{x}^i_{1}+w^i_2\hat{x}^i_{2})$ is a fixed value.
	
	It remains to replace the box constraints $\vel^i\le \vex^i\le \veu^i$ with respect to the new variables.
	
	\smallskip
	\noindent\textbf{Step 2. Deal with the box constraints $\vel^i\le \vex^i\le \veu^i$.}
	
	Plug Eq~\eqref{nn3} and Eq~\eqref{nn2} into the box constraint, we have that
	\begin{eqnarray}
	&& (\ell_h^i-\hat{x}^i_{h}-\xi_h)\le \theta_h(y_i+z_h)\le (u_h^i-\hat{x}^i_{h}-\xi_h), \quad \forall 1\le i\le n, h=1,2\label{eq:a}
	\end{eqnarray}
	To divide the fixed value $\theta_h$ on both sides we need to distinguish between whether it is positive or negative. For simplicity, we define 
	\begin{subequations}
		\begin{eqnarray}
		&&\text{If $\theta_h>0$, then }d^i(\xi_h)=\lceil\frac{\ell_h^i-\hat{x}^i_{h}-\xi_h}{\theta_h}\rceil, \quad \bar{d}^i(\xi_h)=\lfloor\frac{u_h^i-\hat{x}^i_{h}-\xi_h}{\theta_h}\rfloor, \label{eq:theta>0}\\
		&&\text{If $\theta_h<0$, then } d^i(\xi_h)=\lceil\frac{u_h^i-\hat{x}^i_{h}-\xi_h}{\theta_h}\rceil, \quad \bar{d}^i(\xi_h)=\lfloor\frac{\ell_h^i-\hat{x}^i_{h}-\xi_h}{\theta_h}\rfloor. \label{eq:theta<0}
		\end{eqnarray}
	\end{subequations}
	
	Then Eq~\eqref{eq:a} can be simplified as 
	\begin{eqnarray}
	&& d^i(\xi_h)\le y_i+z_h\le \bar{d}^i(\xi_h), \quad \forall 1\le i\le n, h=1,2. \label{eq:newbox}
	\end{eqnarray}
	Here we use the ceiling function to round up the left side and use the floor function to round down the right side since $y_i+z_h$ is an integer. 
	
	We emphasize that here $d^i(\xi_h)$ and $\bar{d}^i(\xi_h)$ are dependent on the variable $\xi_h$, however, since $\xi_h\in [0,|\theta_h|-1]$, either $d^i(\xi_h)$ or $\bar{d}^i(\xi_h)$ may take at most two different values. Hence, a straightforward counting yields $2^{2n}$ possibilities regarding the values for all $d^i(\xi_h)$ and $\bar{d}^i(\xi_h)$. However, notice that $d^i(\xi_h)$'s and $\bar{d}^i(\xi_h)$'s are not independent but change simultaneously as $\xi_h$ changes, we will show that we can divide the range $\xi_h\in [0,|\theta_h|-1]$ into a polynomial number of sub-intervals such that if $\xi_h$ lies in one sub-interval, then all $d^i(\xi_h)$'s and $\bar{d}^i(\xi_h)$'s take some fixed value. We call it an efficient sub-interval.
	
	
	In the following step 3 we will show that $(\text{IP}_4)$ can be solved in FPT time once each $\xi_h$ lies in one of the efficient sub-intervals (and hence all $d^i(\xi_h)$'s and $\bar{d}_i(\xi_h)$'s are fixed), and then in step 4 we prove there are only a polynomial number of different efficient sub-intervals.

	\smallskip
	\noindent\textbf{Step 3. Solve $(\text{IP}_4)$ in FPT time when each $\xi_h$ lies in one efficient sub-interval.}	
	
	For any $h$, let $[\tau_h,\bar{\tau}_h]$ be an arbitrary efficient sub-interval of $\xi_h$ such that all $d^i(\xi_h)$'s and $\bar{d}_i(\xi_h)$'s take fixed value for all $\xi_h\in [\tau_h,\bar{\tau}_h]$. We will handle in Step 4 the construction of each $[\tau_h,\bar{\tau}_h]$. 
	
	From now on we write $d^i(\xi_h)$ and $\bar{d}_i(\xi_h)$ as $d^i_h$ and $\bar{d}^i_h$ as they become fixed values. By Eq~\eqref{eq:newbox} we have
	\begin{eqnarray}
	&& \max\{d^i_1-z_1,d^i_2-z_2\}\le y_i\le \min\{\bar{d}^i_1-z_1,\bar{d}^i_2-z_2\}, \quad \forall 1\le i\le n \label{eq:box1}
	\end{eqnarray}
	Note that among $d^i_1-z_1$ and $d^i_2-z_2$, which one is larger solely depends on $d^i_1-d^i_2$ and $z_1-z_2$. Hence, to get rid of the $\max$ and $\min$ on both sides of Eq~\eqref{eq:box1} for $1\le i\le n$, we need to compare the value of $z_1-z_2$ with at most $2n$ distinct values, which are $d^i_1-d^i_2$ and $\bar{d}^i_1-\bar{d}^i_2$. Now we divide  $(-\infty,\infty)$ into at most $2n+1$ intervals based on the values of $d^i_1-d^i_2$ and $\bar{d}^i_1-\bar{d}^i_2$. Let these intervals be $I_1,I_2,\cdots,I_{2n+1}$. When $z_1-z_2$ lies in one of the intervals, say, $I_k$, Eq~\eqref{eq:box1} can be simplified as
	\begin{eqnarray}
	&& \ell^i(I_k,z_1,z_2)\le y_i\le u^i(I_k,z_1,z_2), \quad \forall 1\le i\le n \label{eq:box2}
	\end{eqnarray}
	where $\ell^i(I_k,z_1,z_2)$ and $u^i(I_k,z_1,z_2)$ are linear functions in $z_1$ and $z_2$. Recall that $y_1=0$, whereas $\ell^1(I_k,z_1,z_2)=u^1(I_k,z_1,z_2)=0$. For simplicity, we define a new variable $p_i:=y_i-\ell^i(I_k,z_1,z_2)$, then it is easy to see that\footnote{This is possible since $\|\vel\|_\infty, \|\veu\|_\infty \le 2^{O(n\log n)}\Delta^{O(n)}$ throughout this paper (see Preliminaries), and thus both the left and right sides are not $\infty$.}
	\begin{eqnarray}
	&& 0\le p_i\le u^i(I_k,z_1,z_2)-\ell^i(I_k,z_1,z_2), \quad \forall 1\le i\le n \label{eq:box3}
	\end{eqnarray}
	
	Now we rewrite $(\text{IP}_4)$ using new variables $p_i$ and $z_1,z_2$ as follows:
	\begin{subequations}
		\begin{eqnarray*}
			(\text{IP}_5[k]):		&\max& \vew\vex=  \vew^0\vex^0+\sum_{h=1}^2\sum_{i=1}^{n}w^i_h\xi_h+\sum_{h=1}^2\sum_{i=1}^{n}w^i_h\theta_h p_i+L(z_1,z_2)  \nonumber\\
			&&	C\vex^0+ D \left(
			\begin{array}{c}
				\sum_{i=1}^{n}\hat{x}^i_{1}+\theta_1\sum_{i=1}^{n}p_i+n\xi_1+L_1(z_1,z_2)\\
				\sum_{i=1}^{n}\hat{x}^i_{2}+\theta_2\sum_{i=1}^{n}p_i +n\xi_2+L_2(z_1,z_2)\\
			\end{array} \right) =\veb^0  \\
			&&	B\vex^0+\lambda\xi_1+\mu\xi_2+\lambda z_1\theta_1+\mu z_2 \theta_2=\veb^1  \\
			&&0\le p_i\le u^i(I_k,z_1,z_2)-\ell^i(I_k,z_1,z_2)  ,\quad \forall 1\le i \le n \\
			&&\xi_h\in[\tau_h,\bar{\tau}_h],\quad h=1,2\\
			&& z_1-z_2\in I_k\\
			&&\vex^0\in\Z^{t_B},\xi_1,\xi_2,z_1,z_2,p_i\in \Z, \quad \forall 1\le i\le n 
		\end{eqnarray*}
	\end{subequations}
	Here $L(z_1,z_2)$, $L_1(z_1,z_2)$, $L_2(z_1,z_2)$ are all linear functions of $z_1,z_2$ (which may contain non-zero constant term). Note again that $p_1$ is a dummy variable as $u^1(I_k,z_1,z_2)=\ell^1(I_k,z_1,z_2)=0$ enforces that $p_1=0$. $(\text{IP}_4)$ can be solved by solving $(\text{IP}_5[k])$ for every $k$ then picking the best solution. 
	
	Now we show how to solve $(\text{IP}_5[k])$. Ignoring the dummy variable $p_1$, a crucial observation is that, while $(\text{IP}_5[k])$ contains variables $p_2,p_3,\cdots,p_n$, they have exactly the same coefficients in constraints, and therefore we can ``merge" them into a single variable $p:=\sum_{i=2}^np_i$. More precisely, we consider the coefficients of $p_i$'s in the objective function, which are $v_i:=\sum_{h=1}^2w_h^i\theta_h$ for $2\le i\le n$. By re-indexing variables, we may assume without loss of generality that $v_2\ge v_3\ge\cdots\ge v_n$. Using a simple exchange argument, we can show that if $p=\sum_{i=2}^np_i\le u^2(I_k,z_1,z_2)-\ell^2(I_k,z_1,z_2)$, then the optimal solution is achieved at $p_2=p$, $p_3=p_4=\cdots=p_n=0$. More generally, if 
	$$\sum_{\gamma=2}^j \left(u^\gamma(I_k,z_1,z_2)-\ell^\gamma(I_k,z_1,z_2)\right)< \sum_{i=2}^np_i\le \sum_{\gamma=2}^{j+1} \left(u^\gamma(I_k,z_1,z_2)-\ell^\gamma(I_k,z_1,z_2)\right),$$
	then the optimal solution is achieved at $p_i=u^i(I_k,z_1,z_2)-\ell^i(I_k,z_1,z_2)$ for $2\le i\le j$ and $p_{i}=0$ for $i>j+1$. 
	
	Define $\Lambda(j):=\sum_{\gamma=2}^j \left(u^\gamma(I_k,z_1,z_2)-\ell^\gamma(I_k,z_1,z_2)\right)$ for $j\ge 2$, $\Lambda(1):=0$, and
	$W(j):=\sum_{h=1}^2\sum_{i=1}^jw_h^i\theta_h\left(u^i(I_k,z_1,z_2)-\ell^i(I_k,z_1,z_2) \right)$. 
	
	Let $(\text{IP}_5[k,j])$ be as follows:
	\begin{subequations}
		\begin{eqnarray*}
			(\text{IP}_5[k,j]):		&\max& \vew\vex=  \vew^0\vex^0+W(j-1)+L(z_1,z_2)  \nonumber\\
			&&\hspace{1cm}+\sum_{h=1}^2\sum_{i=1}^{n}w^i_h\xi_h+\sum_{h=1}^2w^j_h\theta_h \left(p-\Lambda(j-1)\right)\\
			&&	C\vex^0+ D \left(
			\begin{array}{c}
				\sum_{i=1}^{n}\hat{x}^i_{1}+\theta_1p+n\xi_1+L_1(z_1,z_2)\\
				\sum_{i=1}^{n}\hat{x}^i_{2}+\theta_2p+n\xi_2+L_2(z_1,z_2)\\
			\end{array} \right) =\veb^0  \\
			&&	B\vex^0+\lambda\xi_1+\mu\xi_2+\lambda z_1\theta_1+\mu z_2 \theta_2=\veb^1  \\
			&& \Lambda(j-1)<  p\le \Lambda(j) \\
			&&\xi_h\in[\tau_h,\bar{\tau}_h],\quad h=1,2\\
			&& z_1-z_2\in I_k\\
			&&\vex^0\in\Z^{t_B},\xi_1,\xi_2,z_1,z_2,p\in \Z, \quad \forall 1\le i\le n 
		\end{eqnarray*}
	\end{subequations}
	Our argument above shows that $(\text{IP}_5[k])$ can be solved by solving $(\text{IP}_5[k,j])$ for all $1\le j\le n$ and picking the best solution.

	It remains to solve each $(\text{IP}_5[k,j])$. Notice that this is an IP with $O(t_B)$ variables, and thus can be solved in $t_B^{O(t_B)}poly(\log\Delta)$ time by applying Kannan's algorithm. Thus, when each $\xi_h$ lies in one efficient sub-interval, $(\text{IP}_4)$ can be solved in $t_B^{O(t_B)}poly(n,\log\Delta)$ time.
	
	
	
	\smallskip
	\noindent\textbf{Step 4. Bounding the number of efficient sub-intervals of $(\xi_1,\xi_2 ) $.}
	
	Recall Eq~\eqref{eq:theta>0} and Eq~\eqref{eq:theta<0}. For simplicity, we assume $\theta_h>0$, the case of $\theta_h<0$ can be handled in a similar way.
	
	Divide $\ell^i_h-\hat{x}_h^i$ by $\theta_h>0$ and denote by $r_h\in [0,\theta_h-1]$ and $q_h$ the remainder and quotient, respectively. It is easy to see that if $0\le \xi_h < r_h$, then $d^i(\xi_h)=\lceil\frac{\ell_h^i-\hat{x}^i_{h}-\xi_h}{\theta_h}\rceil=q_h+1$. Otherwise, $r_h\le \xi_h<\theta_h$, then $d^i(\xi_h)=\lceil\frac{\ell_h^i-\hat{x}^i_{h}-\xi_h}{\theta_h}\rceil=q_h$. We define $r_h$ as one critical point which distinguishes between $d^i(\xi_h)=q_h+1$ and $d^i(\xi_h)=q_h$.
	
	Similarly, divide $u^i_h-\hat{x}_h^i$ by $\theta_h>0$ and denote by $\bar{r}_h\in [0,\theta_h-1]$ and $\bar{q}_h$ the remainder and quotient, respectively. Using the same argument as above we define $\bar{r}_h$ as one critical point which distinguishes between $\bar{d}^i(\xi_h)=\bar{q}_h$ and $\bar{d}^i(\xi_h)=\bar{q}_h-1$. Critical points can be defined in the same way if $\theta_h<0$.
	
	Overall, we can obtain at most $2n$ distinct critical points for $\xi_h$, which divides the whole interval $(-\infty,\infty)$ into at most $2n+1$ sub-intervals. It is easy to see that once $\xi_h$ lies in one of the sub-interval, all $d^i(\xi_h)$ and $\bar{d}^i(\xi_h)$ take fixed values. 
	
	Since there are at most $(2n+1)^2$ different possibilities regarding the efficient sub-intervals of $\xi_1$ and $\xi_2$, and we have concluded in step 3 that for each possibility $(\text{IP}_4)$ can be solved in $t_B^{O(t_B)} poly(n,\log\Delta)$ time, we know that overall 4-block $n$-fold can be solved in $t_B^{O(t_B)} poly(n,\log\Delta)$ time if $A\in\Z^{1\times 2}$. \qed
\end{proof}



The techniques of Theorem~\ref{n-10} can be further extended to handle the case when $A\in\Z^{s_A\times t_A}$ where $t_A=s_A+1$, $\text{rank}(A)=s_A$. The crucial observation is that, while $\vex^i$ contains more variables, the fact that $\text{rank}(A)=s_A$ and $t_A=s_A+1$ enforces that there can be only one ``free" variable, which is similar to the case when $A\in\Z^{1\times 2}$. Towards this, instead of applying B\'{e}zout's identity in Step 1, we will decompose $A$ into Smith normal form. The following Step 2, 3, 4 are similar except that now there will be $\xi_1,\xi_2,\cdots,\xi_{t_A}$, where each has $2n+1$ efficient sub-intervals. This gives rise to $n^{O(t_A)}$ different possibilities, yielding the overall running time $(t_A+t_B)^{O(t_A+t_B)}n^{O(t_A)}poly(\log\Delta)$. 

\paragraph{Proof of Theorem~\ref{thmm}.}
	Write the constraints of the n-fold IP as follows: 
\begin{subequations}
	\begin{eqnarray}
	&& C\vex^0+D\sum_{i=1}^{n}\vex^i=\veb^0\label{n-24} \\
	&&B\vex^0+A\vex^i=\veb^i,\hspace{31mm}\forall 1\le i \le n\label{n-25} \\
	&&\vel^i \le	\vex^i\le \veu^i, \hspace{38mm}\forall 0\le i \le n\nonumber
	\end{eqnarray}
\end{subequations}
\noindent\textbf{Step 1. Decompose $A$ into Smith normal form to deal with two constraints~\eqref{n-24} and~\eqref{n-25}.}

    From the previous knowledge in Preliminaries, we know that $\bar{A}$ is the Smith normal form of $A$ and $\bar{A}=UAV$, where $U$, $V$ are invertible $s_A\times s_A$ and $(t_A\times t_A)$-matrices. Then $ A=U^{-1}\bar{A}V^{-1}$. One can always calculate the Smith normal form of an integer matrix in polynomial time of $poly(t_A,\log \Delta)$~\cite{kannan1979polynomial}.  
    
   We subtract $B\vex^0+A\vex^1=\veb^1$ from both sides of Eq~\eqref{n-25}, and get the following:
    $$A(\vex^i-\vex^1)=\veb^i-\veb^1. $$
 Let $\vey^i:=V^{-1}(\vex^i-\vex^1)$ and $\widetilde{\veb}^i=U(\veb^i-\veb^1)$, and then we get
	 $\bar{A}\vey^i=\widetilde{\veb}^i.$
	 
	 Assume the diagonal elements of $\bar{A}$ are $\alpha_1,\alpha_2,\ldots, \alpha_{s_A}$. And now we know that $t_A=s_A+1$. Thus, 
	$\alpha_1y^i_1=\widetilde{b}^i_1$, $\alpha_2y^i_2=\widetilde{b}^i_2$, $\cdots$, $\alpha_{s_A}y^i_{s_A}=\widetilde{b}^i_{s_A}$. Actually $\{y^i_h|1\le h\le {s_A},2\le i\le n\}$ are determined uniquely. To be consistent, we introduce dummy variables $y^1_h=0$ for $h=1,2,\ldots,{s_A},t_A$.
	
For $V$ is an invertible $t_A\times t_A$ matrix, $V\vey^i=\vex^i-\vex^1$ and $\vex^i=\vex^1+V\vey^i$.
Thus, 
\begin{eqnarray}
	&&\sum_{i=1}^{n}\vex^i=\sum_{i=1}^{n}\vex^1+V\sum_{i=1}^{n}\vey^i.\label{n-30}
	\end{eqnarray}

Since $\{y^i_h|1\le h\le {s_A},1\le i\le n\}$ are determined uniquely, we can compute $V\sum_{i=1}^{n}\vey^i=({\theta'}_1+\theta_1\sum_{i=1}^{n}y^i_{t_A} ,\ldots, {\theta'}_{t_A}+\theta_{t_A}\sum_{i=1}^{n}y^i_{t_A})$, where ${\theta'}_h$ and $\theta_h$ for all $ h=1,2,\ldots,t_A$ are known integer constants. Plug the above into Eq~\eqref{n-24}, we have
	\begin{eqnarray}
&&	C\vex^0+ D \left(
\begin{array}{c}
{\theta'}_1+\theta_1\sum_{i=1}^{n}y^i_{t_A}+nx_1^1\\
{\theta'}_2+\theta_2\sum_{i=1}^{n}y^i_{t_A}+nx_2^1\\
\vdots\\
{\theta'}_{t_A}+\theta_{t_A}\sum_{i=1}^{n}y^i_{t_A} +nx_{t_A}^1\\
\end{array} \right) =\veb^0.\label{nn12}
\end{eqnarray}
Till now, we have transformed 4-block $n$-fold IP into an equivalent IP with variables $y^i_{t_A}$ and $x^1_h$ for $1\le i\le n$ and $h=1,2,\ldots,t_A$. 

Next, we divide $x^1_h$ by $\theta_h$ and denote by $\xi_h$ and $z_h$ its remainder and quotient, respectively, that is,
\begin{eqnarray}
	&&x_h^1=\xi_h+\theta_h z_h, \quad h=1,2,\ldots,t_A\label{nn44}
	\end{eqnarray}
where $\xi_h\in [0,|\theta_h|-1]$. 

Now we can rewrite the 4-block $n$-fold IP using new variables $\xi_h,z_h$ (where $h=1,2,\ldots,t_A$) and $y^i_{t_A}$ (where $1\le i\le n$). 
\begin{subequations}
	\begin{eqnarray}
(\text{IP}_6):	&\max& \vew\vex=\vew^0\vex^0+c_0+\sum_{i=1}^{n}\sum_{h=1}^{t_A}w^i_h\theta_h(y^i_{t_A}+z_h)+\sum_{i=1}^{n}\sum_{h=1}^{t_A}w^i_h\xi_h\nonumber\\
	&&	C\vex^0+ D \left(
\begin{array}{c}
{\theta'}_1+\theta_1\sum_{i=1}^{n}y^i_{t_A}+n(\xi_1+z_1\theta_1)\\
{\theta'}_2+\theta_2\sum_{i=1}^{n}y^i_{t_A}++n(\xi_2+z_2\theta_2)\\
\vdots\\
{\theta'}_{t_A}+\theta_{t_A}\sum_{i=1}^{n}y^i_{t_A}+n(\xi_{t_A}+z_{t_A}\theta_{t_A})\\
\end{array} \right) =\veb^0\label{nn14}\\
&&B\vex^0+A\left(
\begin{array}{c}
\xi_1+z_1\theta_1 \\
\xi_2+z_2 \theta_2\\
\vdots\\
\xi_{t_A}+z_{t_A}\theta_{t_A}\\
\end{array} \right) =\veb^1\label{n-28}\\
&&	y^1_h=0, \hspace{55mm}\forall 1\le h \le t_A  \\
	&&\vel^i \le	\vex^i\le \veu^i, \hspace{47mm}\forall 0\le i \le n \label{IP6:box}
	\end{eqnarray}
\end{subequations}
where $c_0:=\sum_{i=1}^{n}\sum_{h=1}^{t_A-1}\widetilde{w}^i_hy_h^i$ is a fixed value.

It remains to replace the box constraints $\vel^i\le \vex^i\le \veu^i$ with respect to the new variables.

\smallskip
\noindent\textbf{Step 2. Deal with the box constraints $\vel^i\le \vex^i\le \veu^i$.}

	Plug Eq~\eqref{nn44} and the equality $\vex^i=\vex^1+V\vey^i$, $\forall 1\le i\le n$ into the box constraint,
	we have that
	\begin{eqnarray}
&& \ell^i_h-\widetilde{\theta}^i_h-\xi_h\le{\theta_h}( y^i_{t_A}+z_h)\le u^i_h-\widetilde{\theta}^i_h-\xi_h, \forall 1\le i\le n, h=1,2,\ldots,t_A \label{eq:bb}
\end{eqnarray}
where all $\widetilde{\theta}^i_h$, $1\le h\le t_A$ and $1\le i\le n$, are constants during the computation of $V\vey^i$.

To divide the fixed value $\theta_h$ on both sides we need to distinguish between whether it is positive or negative. Therefore we take the same way with~\eqref{eq:theta>0} and~\eqref{eq:theta<0} in Theorem~\ref{n-10}.

For simplicity, we define 
\begin{subequations}
\begin{eqnarray}
&&\text{If $\theta_h>0$, then }d^i(\xi_h)=\lceil\frac{\ell_h^i-\widetilde{\theta}^i_h-\xi_h}{\theta_h}\rceil, \quad \bar{d}^i(\xi_h)=\lfloor\frac{u_h^i-\widetilde{\theta}^i_h-\xi_h}{\theta_h}\rfloor, \label{eq:theta>}\\
&&\text{If $\theta_h<0$, then } d^i(\xi_h)=\lceil\frac{u_h^i-\widetilde{\theta}^i_h-\xi_h}{\theta_h}\rceil, \quad \bar{d}^i(\xi_h)=\lfloor\frac{\ell_h^i-\widetilde{\theta}^i_h-\xi_h}{\theta_h}\rfloor. \label{eq:theta<}
\end{eqnarray}
\end{subequations}

Then Eq~\eqref{eq:bb} can be simplified as 
\begin{eqnarray}
d^i(\xi_h)\le y^i_{t_A}+z_h\le \bar{d}^i(\xi_h), \quad \forall 1\le i\le n, h=1,2,\ldots,t_A. \label{eq:newboxes}
\end{eqnarray}
Here we use the ceiling function to round up the left side and use the floor function to round down the right side since $ y^i_{t_A}+z_h$ is an integer.

We emphasize that here $d^i(\xi_h)$ and $\bar{d}^i(\xi_h)$ are dependent on the variable $\xi_h$, however, since $\xi_h\in [0,|\theta_h|-1]$, either $d^i(\xi_h)$ or $\bar{d}^i(\xi_h)$ may take at most ${t_A}$ different values. Hence, a straightforward counting yields ${t_A}^{2n}$ possibilities regarding the values for all $d^i(\xi_h)$ and $\bar{d}^i(\xi_h)$. However, notice that $d^i(\xi_h)$'s and $\bar{d}^i(\xi_h)$'s are not independent but change simultaneously as $\xi_h$ changes, we will show that we can divide the range $\xi_h\in [0,|\theta_h|-1]$ into a polynomial number of sub-intervals such that if $\xi_h$ lies in one sub-interval, then all $d^i(\xi_h)$'s and $\bar{d}^i(\xi_h)$'s take some fixed value. We call it an efficient sub-interval.


In the following step 3 we will show that $(\text{IP}_6)$ can be solved in $(t_B+t_A)^{O(t_B+t_A)}poly(\log\Delta)$ time once each $\xi_h$ lies in one of the efficient sub-intervals (and hence all $d^i(\xi_h)$'s and $\bar{d}_i(\xi_h)$'s are fixed), and then in step 4 we prove there are $n^{O(t_A)}$ different efficient sub-intervals. 

\smallskip
\noindent\textbf{Step 3. Solve $(\text{IP}_6)$ in FPT time when each $\xi_h$ lies in one efficient sub-interval.}	

Let $[\tau_h,\bar{\tau}_h]$ be an arbitrary efficient sub-interval of $\xi_h$ such that all $d^i(\xi_h)$'s and $\bar{d}_i(\xi_h)$'s take fixed value for any $\xi_h\in [\tau_h,\bar{\tau}_h]$. From now on we write them as $d^i_h$ and $\bar{d}^i_h$. By Eq~\eqref{eq:newboxes}, $\forall 1\le i\le n$, we have 
	\begin{eqnarray}
	&&\max\{d^i_1-z_1,d^i_2-z_2,\ldots,d^i_{t_A}-z_{t_A}\}\nonumber\\
	&\le& y^i_{t_A}\nonumber\\
	&\le &\min\{\bar{d}^i_1-z_1,\bar{d}^i_2-z_2,\ldots,\bar{d}^i_{t_A}-z_{t_A}\}.\label{eq:box11}
	\end{eqnarray}

When we compare $d^i_{h_1}-z_{h_1}$ and $d^i_{h_2}-z_{h_2}$ for all $1\le i\le n$ and $\forall h_1,h_2\in \{1,2,\ldots, t_A\}$, we just need to compare the value of $z_{h_1}-z_{h_2}$ with at most $2n$ distinct values, which are $d^i_{h_1}-d^i_{h_2}$ and $\bar{d}^i_{h_1}-\bar{d}^i_{h_2}$. Hence, to get rid of the $\max$ and $\min$ on both sides of Eq~\eqref{eq:box11}, we only need to repeat the above process $\frac{t_A(t_A-1)}{2}$ times, creating at most $nt_A(t_A-1)$ critical values, and dividing  $(-\infty,\infty)$ into at most $nt_A(t_A-1)+1$ intervals based on the values of $d^i_{h_1}-d^i_{h_2}$ and $\bar{d}^i_{h_1}-\bar{d}^i_{h_2}$ for all $h_1,h_2\in \{1,2,\ldots, t_A\}$. Let these intervals be $I_1,I_2,\cdots,I_{nt_A(t_A-1)+1}$. When $\{z_{h_1}-z_{h_2}|\forall h_1,h_2\in\{1,2,\ldots,{t_A}\}\}$ belong to one of the intervals, say, $I_k$, Eq~\eqref{eq:box11} can be simplified as
\begin{eqnarray}
\ell^i(I_k,z_1,z_2,\ldots,z_{t_A})\le y^i_{t_A}\le u^i(I_k,z_1,z_2,\ldots,z_{t_A}), \quad \forall 1\le i\le n \label{eq:box12}
\end{eqnarray}
where $\ell^i(I_k,z_1,z_2,\ldots,z_{t_A})$ and $u^i(I_k,z_1,z_2,\ldots,z_{t_A})$ are linear functions in $z_1,z_2$,\\
$\ldots,z_{t_A}$. Recall that $y^1_{t_A}=0$, whereas $\ell^1(I_k,z_1,z_2,\ldots,z_{t_A})=u^1(I_k,z_1,z_2,\ldots,z_{t_A})$\\$=0$. For simplicity, we define a new variable $p_i:=y^i_{t_A}-\ell^i(I_k,z_1,z_2,\ldots,z_{t_A})$, then it is easy to see that
\begin{eqnarray}
0\le p_i\le u^i(I_k,z_1,z_2,\ldots,z_{t_A})-\ell^i(I_k,z_1,z_2,\ldots,z_{t_A}), \quad \forall 1\le i\le n \label{eq:box13}
\end{eqnarray}

Now we rewrite $(\text{IP}_6)$ using new variables $p_i$ and $z_1,z_2,\ldots,z_{t_A}$ as follows:
\begin{subequations}
	\begin{eqnarray*}
(\text{IP}_7[k]):		&\max& \vew\vex=  \vew^0\vex^0+\sum_{h=1}^{t_A}\sum_{i=1}^{n}w^i_h\xi_h+\sum_{h=1}^{t_A}\sum_{i=1}^{n}w^i_h\theta_h p_i+L(z_1,z_2,\ldots,z_{t_A})  \nonumber\\
		&&	C\vex^0+ D \left(
\begin{array}{c}
{\theta'}_1+\theta_1\sum_{i=1}^{n}p_i +n\xi_1+L_1(z_1,z_2,\ldots,z_{t_A})\\
{\theta'}_2+\theta_2\sum_{i=1}^{n}p_i +n\xi_2+L_2(z_1,z_2,\ldots,z_{t_A})\\
\vdots\\
{\theta'}_{t_A}+\theta_{t_A}\sum_{i=1}^{n}p_i+n\xi_{t_A}+L_{t_A}(z_1,z_2,\ldots,z_{t_A})\\
\end{array} \right) =\veb^0\\
&&B\vex^0+A\left(
\begin{array}{c}
\xi_1+z_1\theta_1 \\
\xi_2+z_2 \theta_2\\
\vdots\\
\xi_{t_A}+z_{t_A}\theta_{t_A}\\
\end{array} \right) =\veb^1\\
	&&0\le p_i\le u^i(I_k,z_1,z_2,\ldots,z_{t_A})-\ell^i(I_k,z_1,z_2,\ldots,z_{t_A})  ,\quad \forall 1\le i \le n \\
    &&\xi_h\in[\tau_h,\bar{\tau}_h],\quad h=1,2,\ldots,{t_A}\\
	&& z_{h_1}-z_{h_2}\in I_k, \forall h_1,h_2\in\{1,2,\ldots,{t_A}\}\\
	&&\vex^0\in\Z^{t_B},\xi_h,z_h,p\in \Z, \quad \forall 1\le i\le n,1\le h\le t_A 
	\end{eqnarray*}
	\end{subequations}
	Here $L(z_1,z_2,\ldots,z_{t_A})$, $L_h(z_1,z_2,\ldots,z_{t_A})$, $\forall 1\le h\le t_A$ are all linear functions of $z_1,z_2,\ldots,z_{t_A}$ which may contain constant term. 
	
	Note again that $p_1$ is a dummy variable, $u^1(I_k,z_1,z_2,\ldots,z_{t_A})=\ell^1(I_k,z_1,z_2,$\\$\ldots,z_{t_A})=0$ enforces that $p_1=0$. $(\text{IP}_6)$ can be solved by solving $(\text{IP}_7[k])$ for every $k$ then picking the best solution. 

Now we show how to solve $(\text{IP}_7[k])$. Ignoring the dummy variable $p_1$, a crucial observation is that, while $(\text{IP}_7[k])$ contains variables $p_2,p_3,\cdots,p_n$, they have exactly the same coefficients in constraints, and therefore we can ``merge" them into a single variable $p:=\sum_{i=2}^np_i$. More precisely, we consider the coefficients of $p_i$'s in the objective function, which are $v_i:=\sum_{h=1}^{t_A}w_h^i\theta_h$ for $2\le i\le n$. By re-indexing variables, we may assume without loss of generality that $v_2\ge v_3\ge\cdots\ge v_n$. Using a simple exchange argument, we can show that if $p=\sum_{i=2}^np_i\le u^2(I_k,z_1,z_2,\ldots,z_{t_A})-\ell^2(I_k,z_1,z_2,\ldots,z_{t_A})$, then the optimal solution is achieved at $p_2=p$, $p_3=p_4=\cdots=p_n=0$. More generally, if 
\begin{subequations}
	\begin{eqnarray*}
	&&\sum_{\gamma=2}^j \left(u^\gamma(I_k,z_1,z_2,\ldots,z_{t_A})-\ell^\gamma(I_k,z_1,z_2,\ldots,z_{t_A})\right)\\
	&<& \sum_{i=2}^np_i\\
	&\le& \sum_{\gamma=2}^{j+1} \left(u^\gamma(I_k,z_1,z_2,\ldots,z_{t_A})-\ell^\gamma(I_k,z_1,z_2,\ldots,z_{t_A})\right),
	\end{eqnarray*}
\end{subequations}
then the optimal solution is achieved at $p_i=u^i(I_k,z_1,z_2,\ldots,z_{t_A})-\ell^i(I_k,z_1,z_2,$\\$\ldots,z_{t_A})$ for $2\le i\le j$ and $p_{i}=0$ for $i>j+1$. 

Define $\Lambda(j):=\sum_{\gamma=2}^j \left(u^\gamma(I_k,z_1,z_2,\ldots,z_{t_A})-\ell^\gamma(I_k,z_1,z_2,\ldots,z_{t_A})\right)$, $\Lambda(1):=0$, $W(j):=\sum_{h=1}^{t_A}\sum_{i=1}^jw_h^i\theta_h\left(u^i(I_k,z_1,z_2,\ldots,z_{t_A})-\ell^i(I_k,z_1,z_2,\ldots,z_{t_A}) \right)$. Let $(\text{IP}_7[k,j])$ be as follows:
	\begin{subequations}
	\begin{eqnarray*}
(\text{IP}_7[k,j]):		&\max& \vew\vex=  \vew^0\vex^0+W(j-1)+L(z_1,z_2,\ldots,z_{t_A})  \nonumber\\
&&\hspace{1cm}+\sum_{h=1}^{t_A}\sum_{i=1}^{n}w^i_h\xi_h+\sum_{h=1}^{t_A}w^j_h\theta_h \left(p-\Lambda(j-1)\right)\\
	&&	C\vex^0+ D \left(
\begin{array}{c}
{\theta'}_1+\theta_1p +n\xi_1+L_1(z_1,z_2,\ldots,z_{t_A})\\
{\theta'}_2+\theta_2p +n\xi_2+L_2(z_1,z_2,\ldots,z_{t_A})\\
\vdots\\
{\theta'}_{t_A}+\theta_{t_A}p+n\xi_{t_A}+L_{t_A}(z_1,z_2,\ldots,z_{t_A})\\
\end{array} \right) =\veb^0\\
    &&B\vex^0+A\left(
\begin{array}{c}
\xi_1+z_1\theta_1 \\
\xi_2+z_2 \theta_2\\
\vdots\\
\xi_{t_A}+z_{t_A}\theta_{t_A}\\
    \end{array} \right) =\veb^1\\
	&& \Lambda(j-1)<  p\le \Lambda(j) \\
    &&\xi_h\in[\tau_h,\bar{\tau}_h],\quad h=1,2,\ldots,{t_A}\\
	&& z_{h_1}-z_{h_2}\in I_k,\quad \forall h_1,h_2\in\{1,2,\ldots,{t_A}\}\\
	&&\vex^0\in\Z^{t_B},\xi_h,z_h,p_i\in \Z, \quad \forall 1\le i\le n,1\le h\le t_A 
	\end{eqnarray*}
\end{subequations}
Our argument above shows that $(\text{IP}_7[k])$ can be solved by solving $(\text{IP}_7[k,j])$ for all $1\le j\le n$ and picking the best solution.

It remains to solve each $(\text{IP}_7[k,j])$. Notice that this is an IP with $O(t_A+t_B)$ variables, and thus can be solved in $(t_A+t_B)^{O(t_A+t_B)}poly(n,\log\Delta)$ time by applying Kannan's algorithm. Thus, when each $\xi_h$ lies in one efficient sub-interval, $(\text{IP}_6)$ can be solved in $(t_A+t_B)^{O(t_A+t_B)}poly(n,\log\Delta)$ time.

\smallskip
\noindent\textbf{Step 4. Bounding the number of efficient sub-intervals of $(\xi_1,\xi_2,\cdots,\xi_{t_A})$.}

Recall Eq~\eqref{eq:theta>} and Eq~\eqref{eq:theta<}. For simplicity, we assume $\theta_h>0$, the case of $\theta_h<0$ can be handled in a similar way.

Divide $\ell^i_h-\widetilde{\theta}^i_h$ by $\theta_h>0$ and denote by $r_h\in [0,\theta_h-1]$ and $q_h$ the remainder and quotient, respectively. It is easy to see that if $0\le \xi_h < r_h$, then $d^i(\xi_h)=\lceil\frac{\ell_h^i-\widetilde{\theta}^i_h-\xi_h}{\theta_h}\rceil=q_h+1$. Otherwise, $r_h\le \xi_h<\theta_h$, then $d^i(\xi_h)=\lceil\frac{\ell_h^i-\widetilde{\theta}^i_h-\xi_h}{\theta_h}\rceil=q_h$. We define $r_h$ as one critical point which distinguishes between $d^i(\xi_h)=q_h+1$ and $d^i(\xi_h)=q_h$.

Similarly, divide $u^i_h-\widetilde{\theta}^i_h$ by $\theta_h>0$ and denote by $\bar{r}_h\in [0,\theta_h-1]$ and $\bar{q}_h$ the remainder and quotient, respectively. Using the same argument as above we define $\bar{r}_h$ as one critical point which distinguishes between $\bar{d}^i(\xi_h)=\bar{q}_h$ and $\bar{d}^i(\xi_h)=\bar{q}_h-1$. Critical points can be defined in the same way if $\theta_h<0$.

Overall, we can obtain at most $2n$ distinct critical points for each $\xi_h$, and $2nt_A$ distinct critical points for all $\xi_h$, $\forall 1\le h\le t_A$, which divides the whole interval $(-\infty,\infty)$ into at most $2nt_A+1$ sub-intervals. It is easy to see that once $\xi_h$ lies in one of the sub-interval, all $d^i(\xi_h)$ and $\bar{d}^i(\xi_h)$ take fixed values. Thus, the number of efficient sub-intervals of $(\xi_1,\xi_2,\cdots,\xi_{t_A})$ is $(nt_A)^{O(t_A)}$.
\qed
 
 \  

We remark that the exponential term $n^{O(t_A)}$ comes from the enumeration of all efficient sub-intervals for $\xi_h$'s, where $\xi_h$ is a ``global" variable that appears in constraint~\eqref{eq:bb} for every $1\le i\le n$. If we consider $n$-fold IP and there is no $\vex^0$, then we can get rid of $\xi_h$ and $z_h$ in constraint~\eqref{eq:bb} and derive upper and lower bounds for each $y_i$ directly, yielding the following theorem. 



\begin{theorem}
	If $A\in\Z^{s_A\times t_A} $, $t_A=s_A+1$ and $\text{rank}(A)=s_A$, $n$-fold IP can be solved in linear time of $n\cdot poly(t_A,\log \Delta)$. \label{n-6}
\end{theorem}
\begin{proof}
We write the constraints of $n$-fold IP explicitly as follows: 
 	\begin{subequations}
 	\begin{eqnarray}
 &&D\sum_{i=1}^{n}\vex^i=\veb^0\label{con1}\\
 &&	A\vex^i=\veb^i,  \hspace{34mm}\ \forall 1\le i \le n\label{con2}\\
 &&	\vel^i \le	\vex^i\le \veu^i, \hspace{30mm}\ \forall 1\le i \le n\nonumber
 \end{eqnarray}
 \end{subequations}
 
Let $\bar{A}$ be the Smith normal form of $A$, then there exist integral matrices $U$, $V$, whose inverse are also integral matrices, such that $ A=U^{-1}\bar{A}V^{-1}$.
 Furthermore, $U,V$ can be calculated in time $poly(t_A,\log \Delta)$~\cite{kannan1979polynomial}.  

Combining with Constraint~\eqref{con2}, we have $\bar{A}V^{-1}\vex^i=\widetilde{\veb}^i$, where $\widetilde{\veb}^i=U\veb^i$. Let $\vey^i:=V^{-1}\vex^i$, and in the following we will substitute $\vex$ with new variables $\vey$. Thus we get $\bar{A}\vey^i=\widetilde{\veb}^i$, which implies that $\alpha_jy^i_j=\widetilde{b}^i_j$ for $1\le j\le s_A=t_A-1$. This settles the value of all $y^i_j$'s except $y^i_{t_A}$. 

Next we consider Constraint~\eqref{con1}. It can be written as $D\sum_{i=1}^{n}V\vey^i=\veb^0$. For simplicity let $\widetilde{D}=DV$, then we have $\widetilde{D}\sum_{i=1}^{n}\vey^i=\veb^0$. Note that only $y^i_{t_A}$'s are variables, $\widetilde{D}\sum_{i=1}^{n}\vey^i=\veb^0$ reduces to equalities with only one variable $\sum_{i=1}^ny^i_{t_A}$, which can be solved directly and we get 
$$\sum_{i=1}^ny^i_{t_A}=d_0,$$
for some $d_0$.


Finally we consider the box constraints. From $\vel^i \le	\vex^i\le \veu^i$, we get $\vel^i \le V	\vey^i\le \veu^i$. Recall that the value of all $y^i_j$'s, except $y^i_{t_A}$, has been determined. Hence, $\vel^i \le V	\vey^i\le \veu^i$ reduces to a set of inequalities in $y^i_{t_A}$. Note that each inequality has the form of $\alpha y^i_{t_A}\le \beta$ for some $\alpha$ and $\beta$. Since $y^i_{t_A}$ is an integer, it can be further simplified as $ y^i_{t_A}\le \lfloor \beta/\alpha\rfloor$ if $\alpha>0$, or $ y^i_{t_A}\ge \lceil \beta/\alpha\rceil$ if $\alpha<0$. Hence, $\vel^i \le V	\vey^i\le \veu^i$ can be simplified into the following form:
\begin{eqnarray}
&&	\tilde{\ell}^i\le y^i_{t_A} \le \tilde{u}^i.\label{n-14}
\end{eqnarray}

For ease of discussion, we further substitute $y^i_{t_A}$'s with a new variable $p_i:=y^i_{t_A}-\tilde{\ell}^i$.
Simple calculations show that $\vew\vex=\sum_{i=1}^{n}\vew^i\vex^i=\sum_{i=1}^{n}\vew^iV\vey^i=c_0+\sum_{i=1}^{n}\widetilde{w}^i_{t_A}p_i$ for some $\widetilde{w}^i_{t_A}$ and fixed value $c_0$. Therefore, we can rewrite the $n$-fold IP as:
\begin{eqnarray}
(\text{IP}_8): & \max& c_0+\sum_{i=1}^{n} \widetilde{w}^i_{t_A}p_i\nonumber\\
&&	\sum_{i=1}^{n} p_i=d_0-\sum_{i=1}^n\tilde{\ell}^i \nonumber\\
&&	0 \le p_i \le \tilde{u}^i-\tilde{\ell}^i,\hspace{28mm}\ \forall 1\le i \le n\nonumber
\end{eqnarray}

$(\text{IP}_8)$ can be solved via a simply greedy algorithm. By re-indexing variables, we may assume without loss of generality that $\widetilde{w}^1_{t_A}\ge \widetilde{w}^2_{t_A}\ge \cdots\ge \widetilde{w}^n_{t_A}$. Suppose $\sum_{i=1}^\gamma (\tilde{u}^i-\tilde{\ell}^i)<d_0-\sum_{i=1}^n\tilde{\ell}^i\le \sum_{i=1}^{\gamma+1}(\tilde{u}^i-\tilde{\ell}^i)$, then a simple exchange argument shows that the optimal objective is achieved at $p_j=\tilde{u}^j-\tilde{\ell}^j$ for $1\le j\le \gamma$, $p_{\gamma+1}=d_0-\sum_{i=1}^n\tilde{\ell}^i-\sum_{i=1}^\gamma (\tilde{u}^i-\tilde{\ell}^i)$, and $p_j=0$ for $j\ge \gamma+2$.


 
Overall, the running time is $n\cdot poly(t_A,\log\Delta)$ where $poly(t_A,\log\Delta)$ is the time to compute Smith normal form of $A$.
\end{proof}


\section{Conclusion}
In this paper, we explore the possibility of developing an algorithm that runs polynomially in $\log\Delta$ for block-structured IP. We obtain positive as well as negative results. Our results seem to suggest that the box constraint $\vel\le\vex\le\veu$ significantly impact the tractability. It remains as an important open problem to give a complete characterization on what kind of box constraints may lead to algorithms polynomial in $\log\Delta$. Another interesting open problem is on 4-block $n$-fold IP, when $A\in\Z^{s_A\times t_A} $, $t_A=s_A+1$ and $\text{rank}(A)=s_A$. Currently our algorithm runs in $(t_A+t_B)^{O(t_A+t_B)}\cdot n^{O(t_A)}\cdot poly(\log\Delta)$ time, which is an XP algorithm when taking $t_A,t_B$ as a parameter. It remains open whether there exists an FPT algorithm parameterized by $t_A,t_B$.

 \bibliographystyle{splncs04}
 \bibliography{ipco_block}

\begin{thebibliography}{10}
\providecommand{\url}[1]{\texttt{#1}}
\providecommand{\urlprefix}{URL }
\providecommand{\doi}[1]{https://doi.org/#1}

\bibitem{altmanova2019evaluating}
Altmanov{\'a}, K., Knop, D., Kouteck{\`y}, M.: Evaluating and tuning n-fold
  integer programming. Journal of Experimental Algorithmics (JEA)
  \textbf{24}(1),  1--22 (2019)

\bibitem{chen2020new}
Chen, L., Kouteck{\'{y}}, M., Xu, L., Shi, W.: New bounds on augmenting steps
  of block-structured integer programs. In: Proceedings of the 28th Annual
  European Symposium on Algorithms, (ESA). LIPIcs, vol.~173, pp. 33:1--33:19
  (2020)

\bibitem{chen2018covering}
Chen, L., Marx, D.: Covering a tree with rooted subtrees--parameterized and
  approximation algorithms. In: Proceedings of the 29th Annual ACM-SIAM
  Symposium on Discrete Algorithms (SODA). pp. 2801--2820. SIAM (2018)

\bibitem{cslovjecsek2020n}
Cslovjecsek, J., Eisenbrand, F., Weismantel, R.: N-fold integer programming via
  \text{LP} rounding. arXiv preprint arXiv:2002.07745  (2020)

\bibitem{dadush2011enumerative}
Dadush, D., Peikert, C., Vempala, S.: Enumerative lattice algorithms in any
  norm via \text{M}-ellipsoid coverings. In: 2011 IEEE 52nd Annual Symposium on
  Foundations of Computer Science (FOCS). pp. 580--589. IEEE (2011)

\bibitem{eisenbrand2018faster}
Eisenbrand, F., Hunkenschr{\"o}der, C., Klein, K.M.: Faster algorithms for
  integer programs with block structure. arXiv preprint arXiv:1802.06289
  (2018)

\bibitem{eisenbrand2019algorithmic}
Eisenbrand, F., Hunkenschr{\"o}der, C., Klein, K.M., Kouteck{\`y}, M., Levin,
  A., Onn, S.: An algorithmic theory of integer programming. arXiv preprint
  arXiv:1904.01361  (2019)

\bibitem{eisenbrand2019proximity}
Eisenbrand, F., Weismantel, R.: Proximity results and faster algorithms for
  \text{Integer Programming} using the \text{Steinitz Lemma}. ACM Transactions
  on Algorithms (TALG)  \textbf{16}(1),  1--14 (2019)

\bibitem{faliszewski2018opinion}
Faliszewski, P., Gonen, R., Kouteck{\`y}, M., Talmon, N.: Opinion diffusion and
  campaigning on society graphs. In: IJCAI. pp. 219--225 (2018)

\bibitem{goemans2014polynomiality}
Goemans, M.X., Rothvo{\ss}, T.: Polynomiality for bin packing with a constant
  number of item types. In: Proceedings of the 25th Annual ACM-SIAM symposium
  on Discrete algorithms. pp. 830--839. SIAM (2014)

\bibitem{hemmecke2010polynomial}
Hemmecke, R., K{\"o}ppe, M., Weismantel, R.: A polynomial-time algorithm for
  optimizing over n-fold 4-block decomposable integer programs. In:
  International Conference on Integer Programming and Combinatorial
  Optimization. pp. 219--229. Springer (2010)

\bibitem{hemmecke2013n}
Hemmecke, R., Onn, S., Romanchuk, L.: N-fold integer programming in cubic time.
  Mathematical Programming  \textbf{137}(1-2),  325--341 (2013)

\bibitem{hemmecke2003decomposition}
Hemmecke, R., Schultz, R.: Decomposition of test sets in stochastic integer
  programming. Mathematical Programming  \textbf{94}(2-3),  323--341 (2003)

\bibitem{hoffman2010integral}
Hoffman, A.J., Kruskal, J.B.: Integral boundary points of convex polyhedra. In:
  50 Years of integer programming 1958-2008, pp. 49--76. Springer (2010)

\bibitem{jansen2018empowering}
Jansen, K., Klein, K.M., Maack, M., Rau, M.: Empowering the
  configuration-\text{IP} $-$ new \text{PTAS} results for scheduling with
  setups times. arXiv preprint arXiv:1801.06460  (2018)

\bibitem{jansen2019near}
Jansen, K., Lassota, A., Rohwedder, L.: Near-linear time algorithm for n-fold
  \text{ILPs} via color coding. In: Proceedings of the 46th International
  Colloquium on Automata, Languages, and Programming (ICALP) (2019)

\bibitem{jansen2018integer}
Jansen, K., Rohwedder, L.: On integer programming, discrepancy, and
  convolution. arXiv preprint arXiv:1803.04744  (2018)

\bibitem{kannan1987minkowski}
Kannan, R.: Minkowski's convex body theorem and integer programming.
  Mathematics of Operations Research  \textbf{12}(3),  415--440 (1987)

\bibitem{kannan1979polynomial}
Kannan, R., Bachem, A.: Polynomial algorithms for computing the \text{Smith}
  and \text{Hermite} normal forms of an integer matrix. SIAM Journal on
  Computing  \textbf{8}(4),  499--507 (1979)

\bibitem{karp1972reducibility}
Karp, R.M.: Reducibility among combinatorial problems. In: Complexity of
  Computer Computations, pp. 85--103. Springer (1972)

\bibitem{klein2020complexity}
Klein, K.: About the complexity of two-stage stochastic \text{IPs}. In:
  International Conference on Integer Programming and Combinatorial
  Optimization. pp. 252--265. Springer (2020)

\bibitem{knop2018scheduling}
Knop, D., Kouteck{\`y}, M.: Scheduling meets n-fold integer programming.
  Journal of Scheduling  \textbf{21}(5),  493--503 (2018)

\bibitem{knop2019combinatorial}
Knop, D., Kouteck{\`y}, M., Mnich, M.: Combinatorial n-fold integer programming
  and applications. Mathematical Programming pp. 1--34 (2019)

\bibitem{knop2020voting}
Knop, D., Kouteck{\`y}, M., Mnich, M.: Voting and bribing in single-exponential
  time. ACM Transactions on Economics and Computation (TEAC)  \textbf{8}(3),
  1--28 (2020)

\bibitem{korte2018combinatorial}
Korte, B., Vygen, J.: \text{Combinatorial} \text{Optimization}: \text{Theory}
  and \text{Algorithms}  (2018)

\bibitem{koutecky2018parameterized}
Kouteck{\`y}, M., Levin, A., Onn, S.: A parameterized strongly polynomial
  algorithm for block structured integer programs. In: Proceedings of the 45th
  International Colloquium on Automata, Languages, and Programming (ICALP)
  (2018)

\bibitem{lenstra1983integer}
Lenstra~Jr, H.W.: Integer programming with a fixed number of variables.
  Mathematics of Operations Research  \textbf{8}(4),  538--548 (1983)

\bibitem{papadimitriou1981complexity}
Papadimitriou, C.H.: On the complexity of integer programming. Journal of the
  ACM (JACM)  \textbf{28}(4),  765--768 (1981)

\bibitem{tardos1986strongly}
Tardos, E.: A strongly polynomial algorithm to solve combinatorial linear
  programs. Operations Research  \textbf{34}(2),  250--256 (1986)

\end{thebibliography}

\end{document}